\documentclass{amsart}
\usepackage{amsmath}
\usepackage{enumerate}
\usepackage{amsmath,amsthm,amscd,amssymb}

\usepackage{latexsym}
\usepackage{upref}
\usepackage{verbatim}

\usepackage[mathscr]{eucal}
\usepackage{dsfont}

\usepackage{graphicx}
\usepackage{epstopdf}
\usepackage[colorlinks,hyperindex,pagebackref,hypertex]{hyperref}

\newtheorem{theorem}{Theorem}

\newtheorem{corollary}[theorem]{Corollary}
\newtheorem{lemma}[theorem]{Lemma}
\newtheorem{definition}[theorem]{Definition}

\chardef\bslash=`\\ 

\newcommand{\wh}{\widehat}

\newcommand{\dA}{{\dot A}}

\newcommand{\bbR}{{\mathbb{R}}}

\newcommand{\bbC}{{\mathbb{C}}}

\newcommand{\ran}{\text{\rm{Ran}}}

\newcommand{\dom}{\text{\rm{Dom}}}

\newcommand{\calH}{{\mathcal H}}

\newcommand{\calR}{{\mathcal R}}

\newcommand{\mM}{\mathfrak M}

\def\sM{{\mathfrak M}}   \def\sN{{\mathfrak N}}

\def\bA{{\mathbb A}}      \def\dC{{\mathbb C}}
      \def\dF{{\mathbb F}}

\def\dM{{\mathbb M}}      
      \def\dR{{\mathbb R}}

   \def\cH{{\mathcal H}}

\def\RE{{\rm Re\,}}
\def\Ker{{\rm Ker\,}}

\def\wh{\hat}

\def\uphar{{\upharpoonright\,}}

\DeclareMathOperator{\IM}{Im}

\newcommand{\eval}[2][\right]{\relax
  \ifx#1\right\relax \left.\fi#2#1\rvert}

\begin{document}

\title{On unimodular transformations of conservative L-systems}

\author{S. Belyi}
\address{Department of Mathematics\\ Troy State University\\
Troy, AL 36082, USA\\
}
\curraddr{}
\email{sbelyi@troy.edu}

\author{K. A. Makarov}
\address{Department of Mathematics, University of Missouri, Columbia, MO 63211, USA}
\email{makarovk@missouri.edu}

\author{E. Tsekanovski\u i}
\address{Department of Mathematics,\\ Niagara University, New York
14109\\ USA}
\email{tsekanov@niagara.edu}

\subjclass[2010]{Primary: 81Q10, Secondary: 35P20, 47N50}

\dedicatory{Dedicated with great pleasure to Heinz Langer on the occasion of his 80-th birthday }

\keywords{L-system, transfer function, impedance function,  Herglotz-Nevanlinna function, Weyl-Titchmarsh function, Liv\v{s}ic function, characteristic function,
Donoghue class, symmetric operator, dissipative extension, von Neumann parameter.}

\begin{abstract}
We study unimodular transformations of conservative $L$-systems. Classes $\sM^Q$, $\sM^Q_\kappa$, $\sM^{-1,Q}_\kappa$ that are impedance functions of the corresponding $L$-systems are introduced. A unique unimodular transformation of a given $L$-system with impedance function from the mentioned above classes is found such that the impedance function of a new $L$-system belongs to  $\sM^{(-Q)}$, $\sM^{(-Q)}_\kappa$, $\sM^{-1,(-Q)}_\kappa$, respectively. As a result we get that considered classes (that are perturbations of the Donoghue classes of Herglotz-Nevanlinna functions with an arbitrary real constant $Q$) are invariant under the corresponding unimodular transformations of $L$-systems. We define a coupling of an $L$-system and a so called $F$-system and on its basis obtain a multiplication theorem for their transfer functions. In particular, it is shown
that any unimodular transformation of a given $L$-system is equivalent to a coupling of this system and the corresponding controller, an $F$-system with a constant unimodular transfer function.
In addition, we derive  an explicit form of a controller responsible for a corresponding unimodular transformation of an $L$-system. Examples that illustrate the developed approach are presented.
 \end{abstract}

\maketitle

\section{Introduction}\label{s1}

This paper is yet another part of an ongoing project studying the connections between various subclasses of Herglotz-Nevanlinna functions and conservative realizations of L-systems with one-dimensional input-output space (see  \cite{ABT}, \cite{BMkT}, \cite{BMkT-2}, \cite{MT-S}, \cite{MT10}).

Let $T$ be a densely defined closed operator in a Hilbert space $\cH$ such that its resolvent set $\rho(T)$ is not empty. We also assume that
$\dom(T)\cap \dom(T^*)$ is dense and that the restriction $T|_{\dom(T)\cap \dom(T^*)}$ is a closed symmetric operator with finite equal deficiency indices.
Let $\calH_+\subset\calH\subset\calH_-$ be the rigged Hilbert space associated with $\dot A$.

One of the main objectives  of the current paper is the study of the \textit{L-system}
\begin{equation}
\label{col0}
 \Theta =
\left(%
\begin{array}{ccc}
  \bA    & K & J \\
   \calH_+\subset\calH\subset\calH_- &  & E \\
\end{array}%
\right).
\end{equation}
where the \textit{state-space operator} $\bA$ is a bounded linear operator from
$\calH_+$ into $\calH_-$ such that  $\dA \subset T\subset \bA$, $\dA^* \subset T^* \subset \bA$,
$K$ is a bounded linear operator from the finite-dimensional Hilbert space $E$ into $\calH_-$,  $J=J^*=J^{-1}$ is a self-adjoint isometry  on $E$
such that $\IM\bA=KJK^*$. Due to the facts that $\calH_\pm$ is dual to $\calH_\mp$ and  {that} $\bA^*$ is a bounded linear operator from $\calH_+$ into $\calH_-$, $\IM\bA=(\bA-\bA^*)/2i$ is a well defined bounded operator from $\calH_+$  into $\calH_-$.
Note that the main operator $T$ associated with the system $\Theta$ is uniquely determined by the state-space operator $\bA$ as its restriction onto the domain $\dom(T)=\{f\in\calH_+\mid \bA f\in\calH\}$. A detailed description  of the $L$-systems together with their connections to various subclasses of Herglotz-Nevanlinna functions can be found in \cite{ABT} (see also \cite{AlTs2}, \cite{AlTs3}, \cite{BT3}, \cite{BMkT}, \cite{BMkT-2}, \cite{Bro}).

Recall that the operator-valued function  given by
\begin{equation*}\label{W1}
 W_\Theta(z)=I-2iK^*(\bA-zI)^{-1}KJ,\quad z\in \rho(T),
\end{equation*}
 is called the \textit{transfer function}  of the L-system $\Theta$ and
\begin{equation*}\label{real2}
 V_\Theta(z)=i[W_\Theta(z)+I]^{-1}[W_\Theta(z)-I] =K^*(\RE\bA-zI)^{-1}K,\quad z\in\rho(T)\cap\dC_{\pm},
\end{equation*}
is called the \textit{impedance function } of $\Theta$.

In addition to L-systems we also recall (see \cite{HST3}, \cite{ABT}) the definition of F-systems of the form
\begin{equation*}
 \Theta_{F}=
         \begin{pmatrix}
          M    & F  & K & J  \\
          \calH  & &  & E
         \end{pmatrix},
\end{equation*}
that will play an auxiliary role in our development.

The main  goal of the paper is to study the effect of a \textit{unimodular transformation } applied to an L-system with one-dimensional input-output space. A new twist in our exposition is introducing the concept  of LF-coupling of systems and a controller. Applying the latter to an L-system has an effect equivalent to a corresponding unimodular transformation.

The paper is organized as follows.

In Section \ref{s2}  we recall the definitions of L- and F-systems, their transfer and impedance functions, and provide necessary background.

In Section \ref{s3} we introduce the concept of an LF-coupling that is a coupling of an L-system and an F-system. We also obtain a multiplication theorem of relating transfer functions of LF-coupling and both individual L- and F-system being coupled this way.

In Section \ref{s4} we present the ``perturbed" classes $\sM^Q$, $\sM^Q_\kappa$, and $\sM^{-1,Q}_\kappa$ of impedance functions of L-systems with one-dimensional input-output space.

Section \ref{s5} contains the definition of a unimodular transformation of an L-system of the type considered in Section \ref{s4} and main results of the paper.  Here we construct a unique unimodular transformation of a given $L$-system with impedance function from $\sM^Q$, $\sM^Q_\kappa$, and $\sM^{-1,Q}_\kappa$ classes  such that the impedance function of a new $L$-system belongs to  $\sM^{(-Q)}$, $\sM^{(-Q)}_\kappa$, $\sM^{-1,(-Q)}_\kappa$, respectively.

In Section \ref{s6} we put forward a concept of a controller that is a special form of an F-system with a constant unimodular transfer function. We show that any unimodular transformation of a given $L$-system is equivalent to a coupling of this system with the corresponding controller. In the end of the section we also present an analog of the ``absorbtion property" for the Donoghue class $\sM$ that was discussed in \cite{BMkT-2}.

We conclude the paper by providing several examples that illustrate all the main results and concepts. Connections of the considered systems and the corresponding differential equations are pointed out in Appendix \ref{A1}.

\section{Preliminaries}\label{s2}

For a pair of Hilbert spaces $\calH_1$, $\calH_2$ we denote by $[\calH_1,\calH_2]$ the set of all bounded linear operators from $\calH_1$ to $\calH_2$. Let $\dA$ be a closed, densely defined,
symmetric operator in a Hilbert space $\calH$ with inner product $(f,g),f,g\in\calH$. Any non-symmetric operator $T$ in $\cH$ such that
\[
\dA\subset T\subset\dA^*
\]
is called a \textit{quasi-self-adjoint extension} of $\dA$.

 Consider the rigged Hilbert space (see \cite{Ber}, \cite{BT3})
$\calH_+\subset\calH\subset\calH_- ,$ where $\calH_+ =\dom(\dA^*)$ and
\begin{equation}\label{108}
(f,g)_+ =(f,g)+(\dA^* f, \dA^*g),\;\;f,g \in \dom(A^*).
\end{equation}
Let $\calR$ be the \textit{\textrm{Riesz-Berezansky   operator}} $\calR$ (see \cite{Ber}, \cite{BT3}) which maps $\mathcal H_-$ onto $\mathcal H_+$ such
 that   $(f,g)=(f,\calR g)_+$ ($\forall f\in\calH_+$, $g\in\calH_-$) and
 $\|\calR g\|_+=\| g\|_-$.
 Note that
identifying the space conjugate to $\calH_\pm$ with $\calH_\mp$, we
get that if $\bA\in[\calH_+,\calH_-]$, then
$\bA^*\in[\calH_+,\calH_-].$
An operator $\bA\in[\calH_+,\calH_-]$ is called a \textit{self-adjoint
bi-extension} of a symmetric operator $\dA$ if $\bA=\bA^*$ and $\bA
\supset \dA$.
Let $\bA$ be a self-adjoint
bi-extension of $\dA$ and let the operator $\hat A$ in $\cH$ be defined as follows:
\[
\dom(\hat A)=\{f\in\cH_+:\hat A f\in\cH\}, \quad \hat A=\bA\uphar\dom(\hat A).
\]
The operator $\hat A$ is called a \textit{quasi-kernel} of a self-adjoint bi-extension $\bA$ (see \cite{T69}, \cite{TSh1}, \cite[Section 2.1]{ABT}).
According to the von Neumann Theorem (see \cite[Theorem 1.3.1]{ABT}) the domain of $\wh A$, a self-adjoint extension of $\dA$,  can be expressed as
\begin{equation}\label{DOMHAT}
\dom(\hat A)=\dom(\dA)\oplus(I+U)\sN_{i},
\end{equation}
where $U$ is a $(\cdot)$ (and $(+)$)-isometric operator from $\sN_i$ into $\sN_{-i}$
 and $$\sN_{\pm i}=\Ker (\dA^*\mp i I)$$ are the deficiency subspaces of $\dA$.
 A self-adjoint bi-extension $\bA$ of a symmetric operator $\dA$ is called \textit{t-self-adjoint} (see \cite[Definition 3.3.5]{ABT}) if its quasi-kernel $\hat A$ is
self-adjoint operator in $\calH$.
An operator $\bA\in[\calH_+,\calH_-]$  is called a \textit{quasi-self-adjoint bi-extension} of a non-symmetric  operator $T$ if
$\bA\supset T\supset \dA$ and $\bA^*\supset T^*\supset\dA.$ We will be mostly interested in the following type of quasi-self-adjoint bi-extensions.
\begin{definition}[\cite{ABT}]\label{star_ext}
Let $T$ be a quasi-self-adjoint extension of $\dA$ with nonempty
resolvent set $\rho(T)$. A quasi-self-adjoint bi-extension $\bA$ of an operator $T$ is called a \textit{($*$)-extension }of $T$ if $\RE\bA$ is a
t-self-adjoint bi-extension of $\dA$.
\end{definition}
In what follows we assume that $\dA$ has equal finite deficiency indices and will say that a quasi-self-adjoint extension $T$ of $\dA$ belongs to the
\textit{class $\Lambda(\dA)$} if $\rho(T)\ne\emptyset$, $\dom(\dA)=\dom(T)\cap\dom(T^*)$, and hence  $T$ admits $(*)$-extensions. The description of
all $(*)$-extensions via Riesz-Berezansky   operator $\calR$ can be found in \cite[Section 4.3]{ABT}.

\begin{definition} 
A system of equations
\[
\left\{   \begin{array}{l}
          (\bA-zI)x=KJ\varphi_-  \\
          \varphi_+=\varphi_- -2iK^* x
          \end{array}
\right.,
\]
 or an
array
\begin{equation}\label{e6-3-2}
\Theta= \begin{pmatrix} \bA&K&\ J\cr \calH_+ \subset \calH \subset
\calH_-& &E\cr \end{pmatrix}
\end{equation}
 is called an \textbf{{L-system}}   if:
\begin{enumerate}
\item[(1)] {$\mathbb  A$ is a   ($\ast $)-extension of an
operator $T$ of the class $\Lambda(\dA)$};
\item[(2)] {$J=J^\ast =J^{-1}\in [E,E],\quad \dim E < \infty $};
\item[(3)] $\IM\bA= KJK^*$, where $K\in [E,\calH_-]$, $K^*\in [\calH_+,E]$, and
$\ran(K)=\ran (\IM\bA).$
\end{enumerate}
\end{definition}
In the definition above   $\varphi_- \in E$ stands for an input vector, $\varphi_+ \in E$ is an output vector, and $x$ is a state space vector in
$\calH$.
The operator $\bA$  is called the \textit{state-space operator} of the system $\Theta$, $T$ is the \textit{main operator},  $J$ is the \textit{direction operator}, and $K$ is the  \textit{channel operator}. A system $\Theta$ in \eqref{e6-3-2} is called \textit{minimal} if the operator $\dA$ is a prime operator in $\calH$, i.e., there exists no non-trivial reducing
invariant subspace of $\calH$ on which it induces a self-adjoint operator. 

We  associate with an L-system $\Theta$ the operator-valued function
\begin{equation}\label{e6-3-3}
W_\Theta (z)=I-2iK^\ast (\mathbb  A-zI)^{-1}KJ,\quad z\in \rho (T),
\end{equation}
which is called the \textbf{transfer  function} of the L-system $\Theta$. We also consider the operator-valued function
\begin{equation}\label{e6-3-5}
V_\Theta (z) = K^\ast (\RE\bA - zI)^{-1} K, \quad z\in\rho(\hat A).
\end{equation}
It was shown in \cite{BT3}, \cite[Section 6.3]{ABT} that both \eqref{e6-3-3} and \eqref{e6-3-5} are well defined. The transfer operator-function $W_\Theta (z)$ of the system
$ \Theta $ and an operator-function $V_\Theta (z)$ of the form (\ref{e6-3-5}) are connected by the following relations valid for $\IM z\ne0$, $z\in\rho(T)$,
\begin{equation}\label{e6-3-6}
\begin{aligned}
V_\Theta (z) &= i [W_\Theta (z) + I]^{-1} [W_\Theta (z) - I] J,\\
W_\Theta(z)&=(I+iV_\Theta(z)J)^{-1}(I-iV_\Theta(z)J).
\end{aligned}
\end{equation}
The function $V_\Theta(z)$ defined by \eqref{e6-3-5} is called the
\textbf{impedance function} of an L-system $ \Theta $ of the form
(\ref{e6-3-2}). The class of all Herglotz-Nevanlinna functions in a finite-dimensional Hilbert
space $E$, that can be realized as impedance functions of an L-system, was described in \cite{BT3}, \cite[Definition 6.4.1]{ABT}.

Let $A$ be a closed linear operator in a Hilbert space $\calH$ and let $F$ be an orthogonal projection in $\calH$. Associated to the pair $(A,F)$ is the \textbf{resolvent set} $\rho(A,F)$, i.e.,
the set of all $z \in \dC$ for which $A-zF$ is boundedly invertible in $\calH$ and $(A-zF)^{-1}$ is defined on entire $\calH$. The corresponding \textbf{resolvent operator} is defined as $(A-zF)^{-1}$, $z \in \rho(A,F)$.
Following \cite[Chapter 12]{ABT}, \cite{HST3} we put forward the following
\begin{definition}
\label{system} Let $\calH$ and $E$ be Hilbert spaces with $\dim E<\infty$. 
A system of equations
\begin{equation}\label{e12-18}
    \left\{%
\begin{array}{ll}
    (M-zF)x=KJ\varphi_-, & \hbox{} \\
    \varphi_+=\varphi_- -2iK^* x, & \hbox{} \\
\end{array}%
\right.,\quad z\in\rho(M,F).
\end{equation}
or an  array
\begin{equation}
\label{system1}
 \Theta_{F}=
         \begin{pmatrix}
          M    & F  & K & J  \\
          \calH  & &  & E
         \end{pmatrix},
\end{equation}
is called an \textbf{$F$-system} if:
\begin{enumerate}
\def\labelenumi{\rm (\roman{enumi})}
\item $M\in [\calH,\calH]$;
\item $J=J^*=J^{-1} \in[E,E]$;
\item $\IM M=KJK^*$, where $K\in [E,\calH]$;
\item $F$ is an orthogonal projection in $\calH$;
\item the resolvent sets $\rho(\RE M,F)$ and $\rho(M,F)$ are nonempty.
\end{enumerate}
\end{definition}
To each $F$-system  in Definition \ref{system} one can associate the following \textbf{transfer  function}
\begin{equation}
\label{chfunc}
 W_{\Theta_F}(z)=I-2i K^*(M-zF)^{-1}KJ,  \quad  z \in \rho(M,F),
\end{equation}
and the \textbf{impedance function}
\begin{equation}
\label{Vfunc}
 V_{\Theta_{F}}(z)=K^*(\RE M-zF)^{-1}K,  \quad  z\in \rho(\RE M,F).
\end{equation}

Consider the two $F$-systems $\Theta_{F_1}$ and $\Theta_{F_2}$ of
the form \eqref{system1}, defined by
\begin{equation}
\label{e12-25-nn}
 \Theta_{F_1} =\begin{pmatrix}
          M_1 \hspace{1.5mm} F_1    &K_1  &J  \\
          \calH_1    &   &E
         \end{pmatrix},
\end{equation}
and
\begin{equation}
\label{e12-26-nn}
 \Theta_{F_2} =\begin{pmatrix}
          M_2 \hspace{1.5mm} F_2    &K_2  &J    \\
          \calH_2                     &     &E
         \end{pmatrix}.
\end{equation}
Define the Hilbert space $\calH$ by
\begin{equation}
\label{e12-27-nn}
 \calH=\calH_1 \oplus \calH_2,
\end{equation}
and let $P_j$ be the orthoprojections from $\calH$ onto $\calH_j$,
$j=1,2$. Define the operators $M$, $F$, and $K$ by
\begin{equation}
\label{e12-28-nn}
 M=M_1P_1+M_2P_2+2iK_1JK_2^*P_2,
 \quad
 F=F_1P_1+F_2P_2,
 \quad
 K=K_1+K_2.
\end{equation}
It is shown in \cite[Theorem 12.2.1]{ABT}, \cite{HST3} that if  $\Theta_{F_1}$ is the $F_1$-system in \eqref{e12-25-nn} and let $\Theta_{F_2}$ is the $F_2$-system in \eqref{e12-26-nn}, then the aggregate
\begin{equation}\label{e12-29-nn}
 \Theta =\begin{pmatrix}
          M \hspace{1.5mm} F    &K  &J  \\
          \calH                   &   &E
         \end{pmatrix},
\end{equation}
with $\calH$, $M$, $F$, and $K$, defined by \eqref{e12-27-nn} and \eqref{e12-28-nn}, is also an $F$-system. This $F$-system $\Theta$ in \eqref{e12-29-nn} is called the \textbf{coupling} of  the $F_1$-system $\Theta_{F_1}$ and the $F_2$-system $\Theta_{F_2}$. It is denoted by
\[
 \Theta=\Theta_{F_1}\cdot \Theta_{F_2}.
\]
It is also shown in \cite[Theorem 12.2.2]{ABT}, \cite{HST3} that if an $F$-system $\Theta$ is the coupling of the $F_1$-system $\Theta_{F_1}$ and the $F_2$-system $\Theta_{F_2}$, then  the associated transfer functions satisfy
\begin{equation}\label{e12-30-nn}
 W_\Theta(z)= W_{\Theta_{F_1}}(z) W_{\Theta_{F_2}}(z),
 \quad
 z \in \rho(M_1,F_1) \cap \rho(M_2,F_2).
\end{equation}

\section{Mixed coupling of $L$-systems and $F$-systems}\label{s3}

Consider an $L$-system $\Theta_{L}$ and an $F$-system $\Theta_{F}$ of
the forms \eqref{e6-3-2} and \eqref{system1}, respectively, and  defined by
\begin{equation}
\label{e12-25-n}
 \Theta_{L} =\begin{pmatrix}
          \bA    &K_1  &J  \\
          \calH_{+1} \subset \calH_1 \subset \calH_{-1}    &   &E
         \end{pmatrix},
\end{equation}
and
\begin{equation}\label{e12-26-n}
 \Theta_{F} =\begin{pmatrix}
          M \hspace{1.5mm} F    &K_2  &J    \\
          \calH_2                     &     &E
         \end{pmatrix},
\end{equation}
where $M$ is a bounded in $\calH_2$ operator.
Define the rigged Hilbert space $\calH_+ \subset \calH \subset \calH_-$ by
\begin{equation}
\label{e12-27-n}
 \calH_+ \subset \calH \subset \calH_-=\calH_{+1}\oplus \calH_2 \subset \calH_1\oplus \calH_2 \subset \calH_{-1}\oplus \calH_2.
\end{equation}
Define the operators $\dM\in[\calH_+,\calH_-]$, $\dF:\calH\rightarrow\calH_2$, and $K:E\rightarrow\calH_-$ by
\begin{equation}\label{e12-28-n}
 \dM=\left(
       \begin{array}{cc}
         \bA & 2i K_1 J K_2^* \\
         0 & M \\
       \end{array}
     \right),
 \quad
 \dF=\left(
     \begin{array}{cc}
                I & 0 \\
                0 & F \\
              \end{array}
            \right),
 \quad
 K=\left(
     \begin{array}{c}
       K_1 \\
       K_2 \\
     \end{array}
   \right).
\end{equation}

\begin{definition}\label{t12-1-n}
 Let $\Theta_{L}$ be the $L$-system in \eqref{e12-25-n} and let $\Theta_{F}$ be the $F$-system in \eqref{e12-26-n}. Then the aggregate
\begin{equation}
\label{e12-29-n}
 \Theta_{LF} =\Theta_{L}\cdot\Theta_{F}=\begin{pmatrix}
          \dM \hspace{5.5mm} \dF    &K  &J  \\
          \calH_+ \subset \calH \subset \calH_-                   &   &E
         \end{pmatrix},
\end{equation}
with $\calH_+ \subset \calH \subset \calH_-$, $\dM$, $\dF$, and $K$, defined by \eqref{e12-27-n} and \eqref{e12-28-n}, is called  an \textbf{$LF$-coupling} of systems $\Theta_{L}$  and $\Theta_{F}$.
\end{definition}

Taking adjoints in \eqref{e12-28-n} gives
\begin{equation}\label{e-17-M-star}
 \dM^*=\left(
       \begin{array}{cc}
         \bA^* &  0\\
        -2i K_2 J K_1^*  & M^* \\
       \end{array}
     \right),
 \quad
  K^*=\left(
     \begin{array}{c}
       K_1^* \\
       K_2^* \\
     \end{array}
   \right)^T,\quad
 KJ=\left(
     \begin{array}{c}
       K_1 J \\
       K_2 J\\
     \end{array}
   \right),
\end{equation}
and therefore,
\[
\begin{split}
\dM-\dM^* &=\left(
       \begin{array}{cc}
         \bA-\bA^* &  2i K_1 J K_2^*\\
        2i K_2 J K_1^*  & M-M^* \\
       \end{array}
     \right)
=2i \left(
       \begin{array}{cc}
         K_1JK_1^* &   K_1 J K_2^*\\
         K_2 J K_1^*  & K_2 J K_2^* \\
       \end{array}
     \right)=2iKJK^*.\\
  \end{split}
\]
A function
\begin{equation}\label{couple-trasnfer}
 W_{\Theta_{LF}}(z)=I-2i K^*(\dM-z \dF)^{-1}KJ,  \quad  z \in \rho(\dM,\dF),
\end{equation}
will be associated with $LF$-coupling and called the \textbf{transfer  function of $LF$-coupling}.
\begin{theorem}\label{t12-2-n}
 Let  $\Theta$ be the  $LF$-coupling of an $L$-system $\Theta_{L}$ and the $F$-system $\Theta_{F}$. Then the associated transfer functions satisfy
\begin{equation}
\label{e12-30-n}
 W_{\Theta_{LF}}(z)= W_{\Theta_{L}}(z) W_{\Theta_{F}}(z),  \quad  z \in \rho(T) \cap \rho(M,F).
\end{equation}
\end{theorem}

\begin{proof}
Let $z \in \rho(T)\cap \rho(M,F)$. Observe that
\[
\begin{split}
 \dM-z\dF  &=\left(
       \begin{array}{cc}
         \bA & 2i K_1 J K_2^* \\
         0 & M \\
       \end{array}
     \right)-z\left(
     \begin{array}{cc}
                I & 0 \\
                0 & F  \\
              \end{array}
            \right)=\left(
       \begin{array}{cc}
         \bA-z I & 2i K_1 J K_2^* \\
         0 & M-z F  \\
       \end{array}
     \right), \\
\end{split}
\]
and hence
$$
(\dM-z\dF)^{-1}=\left(
       \begin{array}{cc}
         (\bA-z I)^{-1} & -2i(\bA-z I)^{-1} K_1 J K_2^*(M-z F )^{-1} \\
         0 & (M-z F )^{-1} \\
       \end{array}
     \right).
$$
Indeed, by direct check
\[
\begin{aligned}
(&\dM-z \dF)(\dM-z \dF)^{-1}\\
&=\left(
       \begin{array}{cc}
         \bA-z I & 2i K_1 J K_2^* \\
         0 & M-z F  \\
       \end{array}
     \right)\left(
       \begin{array}{cc}
         (\bA-z I)^{-1} & -2i(\bA-z I)^{-1} K_1 J K_2^*(M-z F )^{-1} \\
         0 & (M-z F )^{-1} \\
       \end{array}
     \right)\\
     &=\left(
               \begin{array}{cc}
                 I & 0 \\
                 0 & I \\
               \end{array}
             \right)=I.
\end{aligned}
\]
Consequently,
$$
\begin{aligned}
(\dM-z \dF)^{-1}K&=\left(
       \begin{array}{cc}
         (\bA-z I)^{-1} & -2i(\bA-z I)^{-1} K_1 J K_2^*(M-z F )^{-1} \\
         0 & (M-z F )^{-1} \\
       \end{array}
     \right)\left(
     \begin{array}{c}
       K_1 \\
       K_2 \\
     \end{array}
   \right)\\
   &=\left(
     \begin{array}{c}
       (\bA-z I)^{-1}K_1  -2i(\bA-z I)^{-1} K_1 J K_2^*(M-z F )^{-1}K_2 \\
       (M-z F )^{-1}K_2 \\
     \end{array}
   \right),
   \end{aligned}
$$
and
$$\begin{aligned}
&K^*(\dM-z \dF)^{-1}K\\
&=(K_1^*\quad K_2^*)\left(
     \begin{array}{c}
      (\bA-z I)^{-1}K_1  -2i(\bA-z I)^{-1} K_1 J K_2^*(M-z F )^{-1}K_2 \\
       (M-z F )^{-1}K_2 \\
     \end{array}
   \right)\\
   &=K_1^*(\bA-z I)^{-1}K_1  -2i(\bA-z I)^{-1} K_1 J K_2^*(M-z F )^{-1}K_2+K_2^*(M-z F )^{-1}K_2.
   \end{aligned}
$$
Furthermore, \eqref{e12-30-n} follows from
\[
\begin{split}
 W_{\Theta_{LF}}(z) & =I-2iK^*(\dM-z \dF)^{-1}KJ \\
 & =I-2i[K_1^*(\bA-z I)^{-1}K_1  -2i(\bA-z I)^{-1} K_1 J K_2^*(M-z F )^{-1}K_2\\
 &+K_2^*(M-z F )^{-1}K_2] \\
 & =[I-2iK_1^*(\bA-zI)^{-1}K_1J][I-2iK_2^*(M-z F )^{-1}K_2J] \\
 & =W_{\Theta_{L}}(z) W_{\Theta_{F}}(z).
\end{split}
\]
\end{proof}
A function
\begin{equation}\label{couple-impedance}
 V_{\Theta_{LF}}(z)=K^*(\RE\dM-z \dF)^{-1}K,  \quad  z \in \rho(\RE\dM,\dF),
\end{equation}
will be associated with $LF$-coupling and called the \textbf{impedance  function of $LF$-coupling}.
First, let us show that the impedance function of $LF$-coupling is well defined. It follows from \eqref{e12-28-n} and \eqref{e-17-M-star} that
$$
\RE\dM-zI=\left(
       \begin{array}{cc}
         \RE\bA-z I & i K_1 J K_2^* \\
         -i K_2 J K_1^* & \RE M-z F  \\
       \end{array}
     \right).
$$
Let $x=\left(
         \begin{array}{c}
           x_1 \\
           x_2 \\
         \end{array}
       \right)$,
where $x_1\in\calH_{+1}$, $x_2\in\calH_2$. Consider an equation
$$
\begin{aligned}
(\RE\dM-zI)x&=\left(
       \begin{array}{cc}
         (\RE\bA-z I)x_1 & i K_1 J K_2^* \\
         -i K_2 J K_1^* & \RE M-z F  \\
       \end{array}
     \right)\left(
         \begin{array}{c}
           x_1 \\
           x_2 \\
         \end{array}
       \right)\\
       &=\left(
       \begin{array}{cc}
         (\RE\bA-z I)x_1 + i K_1 J K_2^*x_2 \\
         -i K_2 J K_1^* x_1+ (\RE M-z F )x_2 \\
       \end{array}
     \right)=\left(
         \begin{array}{c}
           K_1 e \\
           K_2 e\\
         \end{array}
       \right),
 \end{aligned}
$$
for some $e\in E$. Then
$$
\begin{aligned}
(\RE\bA-z I)x_1 + i K_1 J K_2^*x_2&=K_1e,\\
-i K_2 J K_1^* x_1+ (\RE M-z F )x_2&=K_2e.
\end{aligned}
$$
Applying $(\RE\bA-zI)^{-1}$ to the first equation and solving the result for $x_1$ yields
$$
x_1=(\RE\bA-zI)^{-1}[K_1e-i K_1 J K_2^*x_2].
$$
Substituting this value of $x_1$ in to the second equation, we have
 $$
 -i K_2 J K_1^* (\RE\bA-zI)^{-1}[K_1e-i K_1 J K_2^*x_2]+ (\RE M-z F )x_2=K_2e,
 $$
or
$$
[\RE M-z F -K_2 J K_1^*(\RE\bA-zI)^{-1}K_1 J K_2^*]x_2=K_2[I+iJK_1^*(\RE\bA-zI)^{-1}K_1]e.
$$
Taking into account that the impedance function of our L-system $\Theta_L$ is given by
$$
V_{\Theta_L}(z)=K_1^*(\RE\bA-zI)^{-1}K_1,
$$
we have
\begin{equation}\label{e-20}
[\RE M-z F -K_2 JV_{\Theta_L}(z) J K_2^*]x_2=K_2[I+iJV_{\Theta_L}(z)]e.
\end{equation}
Multiplying both sides of \eqref{e-20} by $K_2^*(\RE M-z F )^{-1}$ yields
$$
[K_2^*-K_2^*(\RE M-z F )^{-1}K_2 JV_{\Theta_L}(z) J K_2^*]x_2=K_2^*(\RE M-z F )^{-1}K_2[I+iJV_{\Theta_L}(z)]e.
$$
We recall that
$$
V_{\Theta_F}(z)=K_2^*(\RE M-z F )^{-1}K_2,
$$
and obtain
$$
[I-V_{\Theta_F}(z)JV_{\Theta_L}(z)J]K_2^*x_2=V_{\Theta_F}(z)[I+iJV_{\Theta_L}(z)]e.
$$
Let us assume that in addition to $\rho(\RE M,F )\ne0$ we have that the operator-function $[I-V_{\Theta_F}(z)JV_{\Theta_L}(z)J]$ is invertible at some point $z_0\in\dC_+$. Then applying the theorem on holomorphic operator-function \cite[Appendix 2]{Bro} we have that $[I-V_{\Theta_F}(z)JV_{\Theta_L}(z)J]$ is invertible on the entire $\dC_+$. Then
$$
K_2^* x_2=[I-V_{\Theta_F}(z)JV_{\Theta_L}(z)J]^{-1}V_{\Theta_F}(z)[I+iJV_{\Theta_L}(z)]e.
$$
Consequently, \eqref{e-20} can be modified into
$$
\begin{aligned}
(\RE M-z F )x_2&-K_2 JV_{\Theta_L}(z) J [I-V_{\Theta_F}(z)JV_{\Theta_L}(z)J]^{-1}V_{\Theta_F}(z)[I+iJV_{\Theta_L}(z)]e\\
&=K_2[I+iJV_{\Theta_L}(z)]e,
\end{aligned}
$$
which can be solved for $x_2$ as
$$
\begin{aligned}
x_2&=(\RE M-z F )^{-1}\\
&\times\left(K_2 JV_{\Theta_L}(z) J [I-V_{\Theta_F}(z)JV_{\Theta_L}(z)J]^{-1}V_{\Theta_F}(z)[I+iJV_{\Theta_L}(z)]e \right).
\end{aligned}
$$
Thus, under the assumptions that $\rho(\RE M,F )\ne0$ and $[I-V_{\Theta_F}(z)JV_{\Theta_L}(z)J]$ is invertible at some point $z_0\in\dC_+$, the impedance function $V_{\Theta_{LF}}(z)$ is well defined by \eqref{couple-impedance}.

The impedance function $V_{\Theta_{LF}}(z)$ defined in \eqref{couple-impedance} and the transfer function $W_{\Theta_{LF}}(z)$ defined in \eqref{couple-trasnfer} are closely connected.

\begin{lemma}\label{VW}
Let $\Theta_{LF}$ be an $LF$-coupling of the form \eqref{e12-29-n}. Let also $\rho(\RE M,F )\ne0$ and $[I-V_{\Theta_F}(z)JV_{\Theta_L}(z)J]$ be invertible at some point $z_0\in\dC_+$. Then for all $z\in\rho(\dM,\dF)\cap\rho(\RE \dM,\dF)$
\begin{equation}
\label{trans01}
\begin{split}
 V_{\Theta_{LF}}(z)
  &=i[W_{\Theta_{LF}}(z)-I][W_{\Theta_{LF}}(z)+I]^{-1}J \\
  &=i[W_{\Theta_{LF}}(z)+I]^{-1}[W_{\Theta_{LF}}(z)-I]J,
\end{split}
\end{equation}
and
\begin{equation}
\label{trans02}
\begin{split}
 W_{\Theta_{LF}}(z)
  &=[I-iV_{\Theta_{LF}}(z)J][I+iV_{\Theta_{LF}}(z)J]^{-1} \\
  &=[I+iV_{\Theta_{LF}}(z)J]^{-1}[I-iV_{\Theta_{LF}}(z)J].
\end{split}
\end{equation}
\end{lemma}

\begin{proof}
The following identity with $z\in\rho(\dM,\dF)\cap\rho(\RE \dM,\dF)$
\[
  (\RE \dM-z \dF)^{-1}-(\dM-z \dF)^{-1}  = i(\dM-z \dF)^{-1}\IM \dM (\RE \dM-z \dF)^{-1},
\]
leads to
\[
\begin{split}
 K^*(\RE \dM-z \dF)^{-1}K-
 & K^*(\dM-z \dF)^{-1}K \\
 & = iK^*(\dM- z \dF)^{-1}KJK^*(\RE \dM-z \dF)^{-1}K.
\end{split}
\]
Now in view of \eqref{couple-trasnfer} and \eqref{couple-impedance}
\[
 2V_{\Theta_{LF}}(z)+i(I-W_{\Theta_{LF}}(z))J=(I-W_{\Theta_{LF}}(z))V_{\Theta_{LF}}(z),
\]
or equivalently,
\begin{equation}
\label{e9-22}
 [I+W_{\Theta_{LF}}(z)][I+iV_{\Theta_{LF}}(z)J]=2I.
\end{equation}
Similarly, the identity
\[
  (\RE \dM-z \dF)^{-1}-(\dM-z \dF)^{-1} = i(\RE \dM-z \dF)^{-1}\IM \dM (\dM-z \dF)^{-1}
\]
with $z\in\rho(\dM,\dF)\cap\rho(\RE \dM,\dF)$ leads to
\begin{equation}
\label{eq002}
 [I+iV_{\Theta_{LF}}(z)J][I+W_{\Theta_{LF}}(z)]=2I.
\end{equation}
The equalities \eqref{e9-22} and \eqref{eq002} show that the operators are boundedly invertible and consequently one obtains \eqref{trans01} and \eqref{trans02}.
\end{proof}

It was shown in \cite[Theorem 12.2.4]{ABT}, \cite{ArTs03} that each constant $J$-unitary operator $B$ on a finite-dimensional Hilbert space $E$ can be realized as a transfer  function of some $F$-system of the
form \eqref{system1}. Let us recall the construction of the realizing $F$-system. Assume that $(\pm 1)$ belongs to the resolvent set of the $J$-unitary operator $B$, and define
\[
 C=i[B-I][B+I]^{-1}J.
\]
As it was shown in the proof of \cite[Theorem 12.2.4]{ABT}, $C$ is a self-adjoint operator. Let also $K :\, E \to E$ be any bounded and boundedly invertible operator. Then the aggregate
\begin{equation}\label{e-26-Theta0}
 \Theta_0 =\begin{pmatrix}
          KC^{-1}(I+iCJ)K^* \hspace{1.5mm} 0    &K  &J    \\
          E                   &   &E
         \end{pmatrix},
\end{equation}
is an $F$-system with $F=0$. By construction, $W_{\Theta_0}(z)\equiv B$. Let $\Theta_L$ be an $L$-system of the form \eqref{e12-25-n}. If we compose the $LF$-coupling $\Theta_{L0}$ of $\Theta_L$ and $\Theta_0$ of the form \eqref{e-26-Theta0}
$$
\Theta_{L0}=\Theta_L\cdot\Theta_0,
$$
then according to Theorem \ref{t12-2-n}
\begin{equation}\label{e-27-n}
 W_{\Theta_{L0}}(z)= W_{\Theta_{L}}(z) W_{\Theta_{0}}(z)=W_{\Theta_{L}}(z)B.
\end{equation}
As it was also shown in the proof of \cite[Theorem 12.2.4]{ABT}, the condition of $(\pm 1)\in\rho(B)$ can be released since $E$ is finite-dimensional. In this case it is easy to see that $B$ can be represented in the form
$B=B_1 B_2$, where $B_j$ is a $J$-unitary operator in $E$ and $(\pm 1) \in \rho(B_j)$, $j=1,2$.  Each of the operators $B_1$ and $B_2$ can be realized (see \cite[Theorem 12.2.4]{ABT}) as transfer functions of two $F$-systems $\Theta_{F_1}$ and $\Theta_{F_2}$, respectively, i.e.,
$$
 W_{\Theta_{F_1}}(z)=B_1, \quad W_{\Theta_{F_2}}(z)=B_2.
$$
Consider the coupling $\Theta_F=\Theta_{F_1} \Theta_{F_2}$ of these $F$-systems as defined in \eqref{e12-29-nn} and apply the multiplication formula \eqref{e12-30-nn}. Then
\[
 W_{\Theta_F}(z)=W_{\Theta_{F_1}}(z)W_{\Theta_{F_2}}(z)=B_1B_2=B.
\]

\section{Systems with one-dimensional input-output and Donoghue classes}\label{s4}

In this Section we are going to apply the concepts and results  covered in Section \ref{s3} to $L$- and $F$-systems with one-dimensional input-output space $\dC$. Let
\begin{equation}\label{e-62}
\Theta_L= \begin{pmatrix} \bA&K_1&\ 1\cr  \calH_{+1} \subset \calH_1 \subset \calH_{-1}  & &\dC\cr \end{pmatrix}
\end{equation}
be a minimal scattering $L$-system of the form \eqref{e12-25-n} with one-dimensional input-output space $\dC$ with the main operator $T$ and the quasi-kernel $\hat A$ of $\RE\bA$. 
Let also
\begin{equation}\label{e-63-F}
 \Theta_{F} =\begin{pmatrix}
          M \hspace{1.5mm} F    &K_2  &1    \\
          \calH_2                     &     &\dC
         \end{pmatrix},
\end{equation}
be a minimal $F$-system of the form \eqref{e12-26-n} also with one-dimensional input-output space $\dC$ and $J=1$. Then the $LF$-coupling $\Theta_{LF}=\Theta_{L}\cdot\Theta_{F}$ of the form \eqref{e12-29-n} takes the reduced form
\begin{equation}\label{e-64-LF}
 \Theta_{LF} =\Theta_{L}\cdot\Theta_{F}=\begin{pmatrix}
          \dM \hspace{5.5mm} \dF    &K  &1  \\
          \calH_+ \subset \calH \subset \calH_-                   &   &\dC
         \end{pmatrix}.
\end{equation}
Let us observe that in the case under consideration the conditions of Lemma \ref{VW} can be weakened since $[1-V_{\Theta_F}(z)V_{\Theta_L}(z)]$ is always invertible at some point $z_0\in\dC_+$. Indeed, suppose $z_1\in\dC_+$ is a point where $1-V_{\Theta_F}(z_1)V_{\Theta_L}(z_1)=0$. Then
\begin{equation}\label{e-65}
V_{\Theta_L}(z_1)=\frac{1}{V_{\Theta_F}(z_1)}.
\end{equation}
We know (see \cite{ABT}) that both $V_{\Theta_F}(z)$ and $V_{\Theta_L}(z)$ are Herglotz-Nevanlinna functions mapping $\dC_+$ into itself. Then left hand side of \eqref{e-65} belongs to the upper half-plane while the right hand side clearly must lie in $\dC_-$ which is a contradiction. Therefore $[1-V_{\Theta_F}(z)V_{\Theta_L}(z)]$ is  invertible at any $z\in\dC_+$.

Now we recall the definitions of Donoghue classes of scalar functions (see \cite{BMkT}, \cite{BMkT-2}, \cite{D}).

Denote by $\mM$ the \textbf{Donoghue class} of all analytic mappings $M$ from $\bbC_+$ into itself  that admits the representation (see \cite{D}, \cite{GT}, \cite{KK74})
 \begin{equation}\label{hernev}
M(z)=\int_\bbR \left
(\frac{1}{\lambda-z}-\frac{\lambda}{1+\lambda^2}\right )
d\mu,
\end{equation}
where $\mu$ is an  infinite Borel measure   and
\begin{equation}\label{e-42-int-don}
\int_\bbR\frac{d\mu(\lambda)}{1+\lambda^2}=1\,,\quad\text{equivalently,}\quad M(i)=i.
\end{equation}
We say (see \cite{BMkT}) that an analytic function $M$ from $\bbC_+$ into itself belongs to the \textbf{generalized Donoghue class} $\sM_\kappa$, ($0\le\kappa<1$) if it admits the representation \eqref{hernev}  where $\mu$ is an infinite Borel measure such that
\begin{equation}\label{e-38-kap}
\int_\bbR\frac{d\mu(\lambda)}{1+\lambda^2}=\frac{1-\kappa}{1+\kappa}\,,\quad\text{equivalently,}\quad M(i)=i\,\frac{1-\kappa}{1+\kappa},
\end{equation}
and to the \textbf{generalized Donoghue class} $\sM_\kappa^{-1}$, ($0\le\kappa<1$)  if it admits the representation \eqref{hernev}
and
\begin{equation}\label{e-39-kap}
\int_\bbR\frac{d\mu(\lambda)}{1+\lambda^2}=\frac{1+\kappa}{1-\kappa}\,,\quad\text{equivalently,}\quad M(i)=i\,\frac{1+\kappa}{1-\kappa}.
\end{equation}
Clearly, $\sM_0=\sM_0^{-1}=\sM$, the (standard) Donoghue class introduced above.

It is shown in \cite[Theorem 11]{BMkT} that the impedance function $V_\Theta(z)$ of an L-system $\Theta$ of the form \eqref{e-62} belongs to the class $\sM$ if and only if the von Neumann parameter $\kappa$ of the main operator $T$ of $\Theta$ is zero. Similar descriptions were given to L-systems $\Theta$ whose impedance functions belong to classes $\sM_\kappa$ and $\sM_\kappa^{-1}$ (see \cite[Theorem 12]{BMkT} and \cite[Theorem 5.4]{BMkT-2}).


Let us introduce the ``perturbed" versions of the Donoghue classes above. We say that a scalar Herglotz-Nevanlinna function $V(z)$ belongs to the \textbf{class  $\sM^Q$ }if it admits the following integral representation
 \begin{equation}\label{e-52-M-q}
V(z)= Q+\int_\bbR\left (\frac{1}{\lambda-z}-\frac{\lambda}{1+\lambda^2}\right )d\mu,\quad Q=\bar Q,
\end{equation}
 and has condition \eqref{e-42-int-don}  on the measure $\mu$. Similarly, we introduce perturbed \textbf{classes} $\sM^Q_\kappa$ and $\sM^{-1,Q}_\kappa$ if normalization conditions  \eqref{e-38-kap} and \eqref{e-39-kap}, respectively, hold  on measure $\mu$ in \eqref{e-52-M-q}.

Let us note that it was shown in \cite{ABT} that every function of a Donoghue class mentioned above (standard, generalized, or perturbed)   belongs to the class of Krein-Langer $Q$-functions introduced in \cite{KL}.

\section{A unimodular transformation of an L-system}\label{s5}

Consider an L-system $\Theta$ of the form \eqref{e-62} with a main operator $T$ and transfer function $W_\Theta(z)$. Let $B$ be a complex number such that $|B|=1$. It was shown in \cite[Theorem 8.2.3]{ABT} (see also \cite{ArTs03}) that there exists another L-system $\Theta_B$ of the form \eqref{e-62} with the same main operator $T$ and such that $W_{\Theta_B}(z)=W_\Theta(z)B$. We rely on this result to put forward the following definition.
\begin{definition}\label{d-7}
An L-system $\Theta_\alpha$ is called a \textbf{unimodular transformation} of an L-system $\Theta$ of the form \eqref{e-62} for some $\alpha\in[0,\pi)$ if
\begin{equation}\label{e-35-uni}
    W_{\Theta_\alpha}(z)=W_\Theta(z)\cdot (-e^{2i\alpha}),
\end{equation}
where $W_\Theta(z)$ and $W_{\Theta_\alpha}(z)$ are transfer functions of the corresponding L-systems.
\end{definition}
Note that $\Theta_{\frac{\pi}{2}}=\Theta$.
It is known (see \cite[Theorem 8.3.1]{ABT}) that if $\Theta_\alpha$ is a unimodular transformation of $\Theta$ and  $V_{\Theta_\alpha}(z)$ is its impedance function then
\begin{equation}\label{e-54-frac}
    V_{\Theta_\alpha}(z)=\frac{\cos\alpha+(\sin\alpha)V_\Theta(z)}{\sin\alpha-(\cos\alpha)V_\Theta(z)},\quad z\in\dC_+.
 \end{equation}
The following theorem shows that the class $\sM$ is in some sense invariant under a unimodular transformation.
\begin{theorem}\label{t-9}
 Let  $\Theta_\alpha$ be a unimodular transformation of an L-system $\Theta$ with the impedance function $V_\Theta(z)$ that belongs to class $\sM$. Then $V_{\Theta_\alpha}(z)\in\sM$.
\end{theorem}
\begin{proof}
Since $\Theta_\alpha$ be a unimodular transformation of $\Theta$, then for any $\alpha\in[0,\pi)$ relation \eqref{e-54-frac} takes place. It was shown in \cite[Theorem 8.3.2]{ABT} that in this case the function $V_{\Theta_\alpha}(z)$ admits integral representation \eqref{e-52-M-q}. Thus, all we need to show is that $V_{\Theta_\alpha}(i)=i$. Indeed,
$$
V_{\Theta_\alpha}(i)=\frac{\cos\alpha+(\sin\alpha)V_\Theta(i)}{\sin\alpha-(\cos\alpha)V_\Theta(i)}=\frac{\cos\alpha+(\sin\alpha)i}{\sin\alpha-(\cos\alpha)i}=\frac{1}{-i}=i.
$$
\end{proof}
Now we study how  a unimodular transformation affects the class $\sM^Q$.
\begin{theorem}\label{t-10}
 Let  $\Theta_\alpha$ be a non-trivial ($\alpha\ne\pi/2$) unimodular transformation of an L-system $\Theta$ with the impedance function $V_\Theta(z)$ that belongs to class $\sM^Q$. Then $V_{\Theta_{\alpha}}(z)\in\sM^{-Q}$ if and only if $\tan\alpha={Q}/{2}$.
\end{theorem}
\begin{proof}
Since $V_{\Theta}(z)\in\sM^Q$, then it has integral representation \eqref{e-52-M-q} with $Q\ne0$ and $V_{\Theta}(i)=Q+i$. Then
$$
\begin{aligned}
V_{\Theta_\alpha}(i)&=\frac{\cos\alpha+(\sin\alpha)V_\Theta(i)}{\sin\alpha-(\cos\alpha)V_\Theta(i)}=\frac{\cos\alpha+(\sin\alpha)(Q+i)}{\sin\alpha-(\cos\alpha)(Q+i)}\\
&=\frac{(\cos\alpha+Q\sin\alpha)+i\sin\alpha}{(\sin\alpha-Q\cos\alpha)-i\cos\alpha}=\frac{-Q\cos2\alpha-(1/2)Q^2\sin2\alpha}{(\sin\alpha-Q\sin\alpha)^2+\cos^2\alpha}\\
&+i\,\frac{1}{(\sin\alpha-Q\cos\alpha)^2+\cos^2\alpha}=Q_\alpha+i\int_{\dR}\frac{d\mu_\alpha(\lambda)}{1+\lambda^2}=Q_\alpha+ia_\alpha,
\end{aligned}
$$
where $Q_\alpha$  and  $\mu_\alpha$  are the elements of integral representation \eqref{e-52-M-q} of the function $V_{\Theta_\alpha}(z)$ and $a_\alpha=\int_{\dR}\frac{d\mu_\alpha(\lambda)}{1+\lambda^2}$. Thus,
\begin{equation}\label{e-55-q}
    Q_\alpha=\frac{-Q\cos2\alpha-(1/2)Q^2\sin2\alpha}{(\sin\alpha-Q\cos\alpha)^2+\cos^2\alpha},
\end{equation}
and
\begin{equation}\label{e-56-q-int}
   a_\alpha=\int_{\dR}\frac{d\mu_\alpha(\lambda)}{1+\lambda^2}=\frac{1}{(\sin\alpha-Q\cos\alpha)^2+\cos^2\alpha}.
\end{equation}
If we would like to derive necessary and sufficient conditions on $V_{\Theta_{\alpha}}(z)\in\sM^{-Q}$, then we need to see when $a_\alpha=1$  and $Q_\alpha=-Q$. Setting $a_\alpha=1$ in \eqref{e-56-q-int} yields
$$
(\sin\alpha-Q\cos\alpha)^2+\cos^2\alpha=1,
$$
or
$$
(\sin\alpha-Q\cos\alpha)^2-\sin^2\alpha=0\quad \Leftrightarrow \quad (2\sin\alpha-Q\cos\alpha)\cdot(Q\cos\alpha)=0,
$$
implying that either $Q=0$ or $\alpha=\frac{\pi}{2}$ or $\tan\alpha=Q/2$. Discarding first two options as contradicting to the definition of class $\sM^Q$ or producing trivial transformation, we focus on the third option
\begin{equation}\label{e-41-tan}
    \tan\alpha=\frac{Q}{2}.
\end{equation}
Clearly, under the current set of assumptions, \eqref{e-56-q-int} implies that $a_\alpha=1$ if and only if $\tan\alpha={Q}/{2}$. We observe that in this case \eqref{e-55-q} transforms into
\begin{equation}\label{e-42}
 Q_\alpha={-Q\cos2\alpha-(1/2)Q^2\sin2\alpha}.
\end{equation}

Applying trigonometric identities to \eqref{e-41-tan} yields
$$
\cos^2\alpha=\frac{4}{Q^2+4}\quad\textrm{ and }\quad \sin^2\alpha=\frac{Q^2}{Q^2+4},
$$
and hence
$$
\cos2\alpha=\cos^2\alpha-\sin^2\alpha=\frac{4-Q^2}{Q^2+4}.
$$
Moreover,
$$\cos\alpha=\frac{\pm 2}{\sqrt{Q^2+4}}\quad\textrm{ and }\quad \sin\alpha=\frac{|Q|}{\sqrt{Q^2+4}}.$$
The sign of $\cos\alpha$ above depends on whether $\alpha\in[0,\pi/2)$ (positive) or $\alpha\in(\pi/2,\pi)$ (negative). We also notice that \eqref{e-41-tan} implies that if $Q>0$, then $\alpha\in[0,\pi/2)$ and if $Q<0$, then $\alpha\in(\pi/2,\pi)$. Therefore,
$$
\sin2\alpha=2\sin\alpha\cos\alpha=\frac{\pm 4|Q|}{Q^2+4}=\frac{ 4Q}{Q^2+4}.
$$
Substituting the above values for $\cos2\alpha$ and $\sin2\alpha$ into \eqref{e-42}, we have
$$
Q_\alpha=\frac{-Q(4-Q^2)}{Q^2+4}-\frac{4Q^2 Q}{2(Q^2+4)}=\frac{Q^3-4Q-2Q^3}{Q^2+4}=-\frac{Q(4+Q^2)}{Q^2+4}=-Q.
$$
This completes the proof.
\end{proof}
Let us make one important observation. Clearly, every function $V_1(z)$ of the perturbed class $\sM^Q$ can be represented as
$$
V_1(z)=Q+V_{1,0}(z),
$$
where $V_{1,0}(z)\in\sM$. Theorem \ref{t-10} above shows that for $V_1(z)=V_\Theta(z)\in\sM^Q$ a unimodular transformation with $\tan\alpha=Q/2$ is such that $V_2(z)=V_{\Theta_\alpha}(z)\in\sM^{-Q}$ and hence
$$
V_2(z)=-Q+V_{2,0}(z),
$$
where $V_{2,0}(z)\in\sM$. However, the theorem does not provide a connection between $V_{2,0}(z)$ and $V_{1,0}(z)$ that is not difficult to obtain. Indeed, for $\tan\alpha=Q/2$
\begin{equation}\label{e-49-connection}
\begin{aligned}
V_{2}(z)&=\frac{\cos\alpha+(\sin\alpha)V_1(z)}{\sin\alpha-(\cos\alpha)V_1(z)}=\frac{1+(\tan\alpha) V_1(z)}{\tan\alpha-V_1(z)}=\frac{1+(Q/2) V_1(z)}{Q/2-V_1(z)}\\
&=\frac{2+Q V_1(z)}{Q-2V_1(z)}=\frac{2+Q (Q+V_{1,0}(z))}{Q-2(Q+V_{1,0}(z))}=-\frac{2+Q^2+ QV_{1,0}(z)}{Q+2V_{1,0}(z)}\\
&=-Q+\frac{QV_{1,0}(z)-2}{Q+2V_{1,0}(z)}.
\end{aligned}
\end{equation}
A direct substitution into the above formula yields that $V_2(i)=-Q+i$ which immediately confirms  that $V_{2,0}(z)\in\sM$. Thus, we have established a formula relating $V_{2,0}(z)$ and $V_{1,0}(z)$
\begin{equation}\label{e-50-connection}
    V_{2,0}(z)=\frac{QV_{1,0}(z)-2}{Q+2V_{1,0}(z)}.
\end{equation}

A similar to Theorem \ref{t-10} result takes place for the other two classes $\sM^Q_\kappa$ and $\sM^{-1,Q}_\kappa$.
\begin{theorem}\label{t-11}
 Let  $\Theta_\alpha$ be a non-trivial ($\alpha\ne\pi/2$) unimodular transformation of an L-system $\Theta$ with the impedance function $V_\Theta(z)$ that belongs to class $\sM^Q_\kappa$. Then $V_{\Theta_{\alpha}}(z)\in\sM^{-Q}_\kappa$ if and only if
 \begin{equation}\label{e-43}
 \tan\alpha=\frac{b}{2Q},
 \end{equation}
 where
 \begin{equation}\label{e-44}
    b=Q^2+a^2-1\quad\textrm{ and }\quad a=\frac{1-\kappa}{1+\kappa}.
 \end{equation}

\end{theorem}
\begin{proof}
Since $V_{\Theta}(z)\in\sM^Q_\kappa$, then it has integral representation \eqref{e-52-M-q} with $Q\ne0$ and $V_{\Theta}(i)=Q+ai$, where $a$ is defined in \eqref{e-44}. Then
$$
\begin{aligned}
V_{\Theta_\alpha}(i)&=\frac{\cos\alpha+(\sin\alpha)V_\Theta(i)}{\sin\alpha-(\cos\alpha)V_\Theta(i)}=\frac{\cos\alpha+(\sin\alpha)(Q+ai)}{\sin\alpha-(\cos\alpha)(Q+ai)}\\
&=\frac{(\cos\alpha+Q\sin\alpha)+ia\sin\alpha}{(\sin\alpha-Q\cos\alpha)-ia\cos\alpha}=\frac{(1/2)(1-Q^2-a^2)\sin2\alpha-Q\cos2\alpha}{(\sin\alpha-Q\sin\alpha)^2+a^2\cos^2\alpha}\\
&+i\,\frac{a}{(\sin\alpha-Q\cos\alpha)^2+a^2\cos^2\alpha}=Q_\alpha+i\int_{\dR}\frac{d\mu_\alpha(\lambda)}{1+\lambda^2}=Q_\alpha+ia_\alpha,
\end{aligned}
$$
where $Q_\alpha$  and  $\mu_\alpha$  are the elements of integral representation \eqref{e-52-M-q} of the function $V_{\Theta_\alpha}(z)$ and $a_\alpha=\int_{\dR}\frac{d\mu_\alpha(\lambda)}{1+\lambda^2}$.
Thus,
\begin{equation}\label{e-55-q'}
    Q_\alpha=\frac{(1/2)(1-Q^2-a^2)\sin2\alpha-Q\cos2\alpha}{(\sin\alpha-Q\cos\alpha)^2+a^2\cos^2\alpha},
\end{equation}
and
\begin{equation}\label{e-56-q-int'}
   a_\alpha=\frac{a}{(\sin\alpha-Q\cos\alpha)^2+a^2\cos^2\alpha}.
\end{equation}
If we would like to derive necessary and sufficient conditions on $V_{\Theta_{\alpha}}(z)\in\sM^{-Q}_\kappa$, then we need to see when $a_\alpha=a$  and $Q_\alpha=-Q$. Setting $a_\alpha=a$ in \eqref{e-56-q-int} yields
$$
(\sin\alpha-Q\cos\alpha)^2+a\cos^2\alpha=1,
$$
or
$$
\sin\alpha-2Q\sin\alpha\cos\alpha+Q^2\cos^2\alpha+a^2\cos^2\alpha=1,
$$
that is equivalent to
$$
(Q^2+a^2-1)\cos^2\alpha-2Q\sin\alpha\cos\alpha=0.
$$
Using \eqref{e-44} we get
$$
\cos\alpha(b\cos\alpha-2Q\sin\alpha)=0.
$$
Since $\alpha\ne\pi/2$ by the condition of our theorem, then we have
$$
b\cos\alpha-2Q\sin\alpha=0,
$$
or $\tan\alpha=\frac{b}{2Q}$. Thus we have just proven that \eqref{e-43} is equivalent to $a_\alpha=a$. All we need to show than that in the case when \eqref{e-43} holds, $Q_\alpha=-Q$.
We observe that if  $a_\alpha=a$, \eqref{e-55-q'} transforms into
\begin{equation}\label{e-47}
 Q_\alpha=(1/2)(1-Q^2-a^2)\sin2\alpha-Q\cos2\alpha=-\frac{b}{2}\sin2\alpha-Q\cos2\alpha.
\end{equation}
Applying trigonometric identities to \eqref{e-43} yields
$$
\cos^2\alpha=\frac{4Q^2}{4Q^2+b^2}\quad\textrm{ and }\quad \sin^2\alpha=\frac{b^2}{4Q^2+b^2},
$$
and hence
$$
\cos2\alpha=\cos^2\alpha-\sin^2\alpha=\frac{4Q^2-b^2}{4Q^2+b^2}.
$$
Moreover,
\begin{equation}\label{e-48}
\cos\alpha=\frac{ 2|Q|}{\sqrt{4Q^2+b^2}}\quad\textrm{ and }\quad \sin\alpha=\frac{|b|}{\sqrt{4Q^2+b^2}}.
\end{equation}
Assume that $\alpha\in(0,\pi/2)$. Then $\tan\alpha>0$ and \eqref{e-43} implies that $|b/2Q|>0$ which means that either: (i) $b>0$ and $Q>0$ or (ii) $b<0$ and $Q<0$. Since both $\cos\alpha$ and $\sin\alpha$ are positive in the first quadrant, then  \eqref{e-48} will turn into
\begin{equation}\label{e-49}
\cos\alpha=\frac{ \pm 2Q}{\sqrt{4Q^2+b^2}}\quad\textrm{ and }\quad \sin\alpha=\frac{\pm b}{\sqrt{4Q^2+b^2}},
\end{equation}
where $(+)$ sign in both formulas is taken in the case (i) and $(-)$ sign, respectively, in the case of (ii).

Now assume that $\alpha\in(\pi/2,\pi)$. Then $\tan\alpha>0$ and \eqref{e-43} implies that $|b/2Q|<0$ which means that either: (iii) $b>0$ and $Q<0$ or (iv) $b<0$ and $Q>0$. But this time we are in the second quadrant and hence $\cos\alpha<0$ while $\sin\alpha>0$. Consequently, formula \eqref{e-49} is true again in the sense that $(+)$ sign in both formulas is taken in the case (iii) and $(-)$ sign in the case (iv). Thus in all the possible cases (i)--(iv) the signs in the numerators in \eqref{e-49} match.

We have then
$$
\begin{aligned}
Q_\alpha&=-\frac{b}{2}\sin2\alpha-Q\cos2\alpha=-b\sin\alpha\cos\alpha-Q\cos2\alpha\\
&=-\frac{2b(\pm Q)(\pm b)|}{4Q^2+b^2}-\frac{Q(4Q^2-b^2)}{4Q^2+b^2}=-\frac{2b^2Q+Q(4Q^2-b^2)}{4Q^2+b^2}\\
&=(-Q)\,\frac{2b^2+4Q^2-b^2}{4Q^2+b^2}=-Q.
\end{aligned}
$$
This completes the proof.
\end{proof}

A similar result takes place for the class $\sM^{-1,Q}_\kappa$.
\begin{theorem}\label{t-12}
 Let  $\Theta_\alpha$ be a non-trivial ($\alpha\ne\pi/2$) unimodular transformation of an L-system $\Theta$ with the impedance function $V_\Theta(z)$ that belongs to class $\sM^{-1,Q}_\kappa$. Then $V_{\Theta_{\alpha}}(z)\in\sM^{-1,-Q}_\kappa$ if and only if \eqref{e-43} holds true for
  \begin{equation}\label{e-50}
    b=Q^2+a^2-1\quad\textrm{ and }\quad a=\frac{1+\kappa}{1-\kappa}.
 \end{equation}
 \end{theorem}
\begin{proof}
The proof has similar to the one of Theorem \ref{t-11} structure. Performing the same set of derivations as we did in the proof of Theorem \ref{t-11}  we show that \eqref{e-43} holds if and only if $a_\alpha=a$. The main difference in what follows is that since $V_{\Theta}(z)\in\sM^{-1,Q}_\kappa$, then $a>1$ and consequently $b>0$ for any real $Q$. As a result, if we assume that
$\alpha\in(0,\pi/2)$, then we can immediately conclude that $Q>0$ or otherwise we will arrive at a contradiction to $\tan\alpha>0$ in the first quadrant.  Similarly, the assumption $\alpha\in(\pi/2,\pi)$ yields $Q<0$. Consequently, \eqref{e-49} becomes
\begin{equation}\label{e-51}
\cos\alpha=\frac{ 2Q}{\sqrt{4Q^2+b^2}}\quad\textrm{ and }\quad \sin\alpha=\frac{b}{\sqrt{4Q^2+b^2}},
\end{equation}
for any $\alpha\in(0,\pi/2)\cup(\pi/2,\pi)$. Evaluating $Q_\alpha$ as we did in the proof of Theorem \ref{t-11} we obtain
$$
\begin{aligned}
Q_\alpha&=-\frac{b}{2}\sin2\alpha-Q\cos2\alpha=-b\sin\alpha\cos\alpha-Q\cos2\alpha\\
&=-\frac{2b^2Q}{4Q^2+b^2}-\frac{Q(4Q^2-b^2)}{4Q^2+b^2}=-Q.
\end{aligned}
$$
Thus, $V_{\Theta_{\alpha}}(z)\in\sM^{-1,-Q}_\kappa$ and the proof is complete.
\end{proof}
We make another observation similar to the one we made after Theorem \ref{t-10}. Clearly, every function $V_1(z)$ of the perturbed class $\sM^{Q}_\kappa$ (or $\sM^{-1,Q}_\kappa$) can be written as
$$
V_1(z)=Q+V_{1,0}(z),
$$
where $V_{1,0}(z)\in\sM_\kappa$ (or $V_{1,0}(z)\in\sM^{-1}_\kappa$). Theorems \ref{t-11} and \ref{t-12}  show that for $V_1(z)=V_\Theta(z)\in\sM^{Q}_\kappa$ (or $V_1(z)=V_\Theta(z)\in\sM^{-1,Q}_\kappa$) a unimodular transformation with $\tan\alpha=b/2Q$ is such that $V_2(z)=V_\Theta(z)\in\sM^{-Q}$ (or $V_2(z)=V_\Theta(z)\in\sM^{-1,-Q}$) and hence
$$
V_2(z)=-Q+V_{2,0}(z),
$$
where $V_{2,0}(z)\in\sM_\kappa$ (or $V_{2,0}(z)\in\sM^{-1}_\kappa$). However, the theorems do not provide a connection between $V_{2,0}(z)$ and $V_{1,0}(z)$ that is not difficult to obtain. Following \eqref{e-49-connection} for $\tan\alpha=b/2Q$ we get
\begin{equation}\label{e-60-connection}
\begin{aligned}
V_{2}(z)&=\frac{\cos\alpha+(\sin\alpha)V_1(z)}{\sin\alpha-(\cos\alpha)V_1(z)}=\frac{1+(\tan\alpha) V_1(z)}{\tan\alpha-V_1(z)}=\frac{1+(b/2Q) V_1(z)}{b/2Q-V_1(z)}\\
&=\frac{2Q+b V_1(z)}{b-2QV_1(z)}=\frac{2Q+b (Q+V_{1,0}(z))}{b-2Q(Q+V_{1,0}(z))}=-\frac{2Q+bQ+ bV_{1,0}(z)}{2Q^2+2QV_{1,0}(z)-b}\\
&=-Q+\frac{Q^3+Q^2V_{1,0}(z)-bQ-Q-(b/2)V_{1,0}(z)}{Q^2+QV_{1,0}(z)-(b/2)}.
\end{aligned}
\end{equation}
Thus, we have established a formula relating $V_{2,0}(z)$ and $V_{1,0}(z)$
\begin{equation}\label{e-61-connection}
    V_{2,0}(z)=\frac{Q^3+Q^2V_{1,0}(z)-bQ-Q-(b/2)V_{1,0}(z)}{Q^2+QV_{1,0}(z)-(b/2)}.
\end{equation}

The result below immediately follows from Theorems \ref{t-10}--\ref{t-12}.
\begin{corollary}\label{c-13}
Let $\Theta$ be an L-system of the form \eqref{e-62} with the impedance function $V_\Theta(z)$. Then there exists a unique (for a given $Q$) unimodular transformation $\Theta_\alpha$ of $\Theta$ such that its impedance function $V_{\Theta_{\alpha}}(z)$ belongs to exactly one of the disjoint classes $\sM^{-Q}$, $\sM^{-Q}_\kappa$, or $\sM^{-1,-Q}_\kappa$.
\end{corollary}

\section{Control of L-systems}\label{s6}

In this section we are going to formalize the procedure of unimodular transformation of an L-system. We start off with the following definition.
\begin{definition}\label{d-14}
An L-system $\Theta$ of the form \eqref{e-62} is called  \textbf{equivalent} to an LF-system $\Theta_{LF}$ of the form \eqref{e-64-LF} if the transfer mappings $W_\Theta(z)$ and $W_{\Theta_{LF}}(z)$ of both systems coincide on the intersection of their domains of definitions.
\end{definition}
In Section \ref{s4} we mentioned that any constant $J$-unitary operator $B$ on a finite-dimensional Hilbert space $E$ can be realized as a transfer function of an F-system $\Theta_0$ of the form \eqref{e-26-Theta0}. Now we apply this result to the situation treated in Section \ref{s5}. We set
$$
B=-e^{2i\alpha},\quad E=\dC,\quad J=1, \quad \alpha\in\left(0,\frac{\pi}{2}\right)\cup\left(\frac{\pi}{2},\pi\right).
$$
Then the operator $C$ involved in the construction of $\Theta_0$ is
\[
 C=i[B-I][B+I]^{-1}J=i\frac{-e^{2i\alpha}-1}{-e^{2i\alpha}+1}=i\frac{e^{i\alpha}+e^{-i\alpha}}{e^{i\alpha}-e^{-i\alpha}}=\cot\alpha.
\]
Also, the main operator of the F-system $\Theta_0$ of the form \eqref{e-26-Theta0} is
$$
KC^{-1}(I+iCJ)K^*=K(C^{-1}+i)K^*=K(\tan\alpha+i)K^*.
$$
By construction, the operator $K$ in  F-system $\Theta_0$ can be chosen as any bounded and boundedly invertible operator from $E$ to $E$. In our case $E=\dC$ and hence we can chose $K=1$. As a result, the F-system $\Theta_0$ of the form \eqref{e-26-Theta0} in our case boils down to
\begin{equation}\label{e-52-Theta0}
 \Theta_{0,\alpha} =\begin{pmatrix}
          \tan\alpha+i \hspace{2mm} 0    &1  &1    \\
          \dC                   &   &\dC
         \end{pmatrix},\quad \alpha\in\left(0,\frac{\pi}{2}\right)\cup\left(\frac{\pi}{2},\pi\right).
\end{equation}
We know that $W_{\Theta_{0,\alpha}}(z)\equiv -e^{2i\alpha}$.

In the case when $\alpha=\pi/2$, $B=1$ and parameter $C^{-1}$ is undefined. We utilize the approach explained in Section \ref{s5}. Namely, we represent
$$B=1=(-i)(i)=(-e^{2i\cdot\frac{\pi}{4}})\cdot (-e^{2i\cdot\frac{3\pi}{4}})=B_1\cdot B_2.$$
The corresponding $C_1=\cot\frac{\pi}{4}=1$ and  $C_2=\cot \frac{3\pi}{4}=-1$ and
\begin{equation}\label{e-53-Theta0}
 \Theta_{0,\frac{\pi}{4}} =\begin{pmatrix}
          1+i \hspace{2mm} 0    &1  &1    \\
          \dC                   &   &\dC
         \end{pmatrix},\quad
 \Theta_{0,\frac{3\pi}{4}} =\begin{pmatrix}
          -1+i \hspace{2mm} 0    &1  &1    \\
          \dC                   &   &\dC
         \end{pmatrix},
\end{equation}
with $W_{\Theta_{0,\frac{3\pi}{4}}}(z)\equiv -i$  and $W_{\Theta_{0,\frac{\pi}{4}}}(z)\equiv i$ are F-systems of the form \eqref{e-26-Theta0} that realize $B_1$ and $B_2$.

Similarly, in the case when $\alpha=0$, $B=-1$ and parameter $C$ is undefined. We proceed as above and represent
$$B=-1=i^2=(-e^{2i\cdot\frac{3\pi}{4}})\cdot (-e^{2i\cdot\frac{3\pi}{4}})=B_2\cdot B_2.$$
The corresponding $C_2=\cot\frac{3\pi}{4}=-1$  and $ \Theta_{0,\frac{3\pi}{4}}$ is given by \eqref{e-53-Theta0}.
\begin{definition}\label{d-15}
An F-system $\Theta_{0,\alpha}$ of the form \eqref{e-52-Theta0} is called a \textbf{controller} to an L-system $\Theta_L$ of the form \eqref{e-62} corresponding to a unimodular transformation $\Theta_\alpha$ for $\alpha\in\left(0,\frac{\pi}{2}\right)\cup\left(\frac{\pi}{2},\pi\right)$.
\end{definition}
In ``trivial" cases when $\alpha=0$ and $\alpha=\pi/2$ the controller is respectively defined as a coupling of the corresponding F-systems
\begin{equation}\label{e-54-cont}
    \Theta_{0,0} =\Theta_{0,\frac{\pi}{4}}\cdot\Theta_{0,\frac{3\pi}{4}}\quad\textrm{ and }\quad\Theta_{0,\frac{\pi}{2}}=\Theta_{0,\frac{3\pi}{4}}\cdot\Theta_{0,\frac{3\pi}{4}}.
\end{equation}

The following result follows directly from the above discussion.
\begin{theorem}\label{t-16}
Let $\Theta_{LF}$ be an LF-coupling of an L-system $\Theta_L$  of the form \eqref{e-62} and a controller $\Theta_{0,\alpha}$ for $\alpha\in[0,\pi)$, that is
$$
\Theta_{LF}=\Theta_L\cdot \Theta_{0,\alpha}.
$$
 Then $\Theta_{LF}$ is equivalent to a unimodular transformation $\Theta_\alpha$ of $\Theta_L$ for the same value of $\alpha$ and hence $W_{\Theta_{LF}}(z)=W_{\Theta_{\alpha}}(z)$ on the intersection of their domains of definitions.
\end{theorem}
Theorem \ref{t-16} is illustrated on Figure \ref{fig-1}. The following theorem is an analogue of the ``absorbtion property" of the class $\sM$ that was discussed in details in \cite{BMkT-2}.
\begin{theorem}\label{t-17}
Let $\Theta_{L}$ be an L-system of the form \eqref{e-62} such that $V_{\Theta_L}\in\sM$ and let $\Theta_{0,\alpha}$ be a controller with an arbitrary value of $\alpha\in[0,\pi)$.
If $\Theta_{LF}$ is an LF-coupling such that $\Theta_{LF}=\Theta_L\cdot \Theta_{0,\alpha}$,  then  $V_{\Theta_{LF}}(z)\in\sM$.
\end{theorem}
\begin{proof}
 The proof of this result follows from the invariance of the Donoghue class $\sM$ under a unimodular transformation (see \cite{BMkT}, \cite{BMkT-2}, \cite{ABT}) and Theorem \ref{t-16}.
\end{proof}

\begin{figure}
  \begin{center}
  \includegraphics[width=100mm]{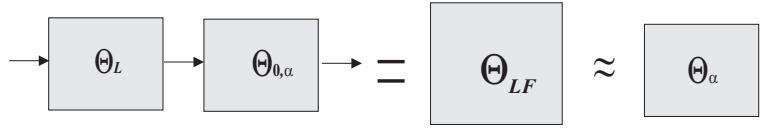}
  \caption{Applying a controller}\label{fig-1}
  \end{center}
\end{figure}

\section{Examples}\label{s7}

\subsection*{Example 1}\label{ex-1}
Consider an L-system
\begin{equation}\label{e8-61}
\Theta^{(\xi)}=
\begin{pmatrix}
\bA^{(\xi)}&K^{(\xi)} &1\\
&&\\
W_1^2\subset L^2_{[0,l]}\subset (W^1_2)_- &{ } &\dC
\end{pmatrix},
\end{equation}
 where
$$\begin{aligned}
\bA^{(\xi)} x&=\frac{1}{i}\frac{dx}{dt}+ix(l)\left[\delta(t-l)-e^{-i\xi l}\delta(t)\right],\\
{\bA^{(\xi)}}^* x&=\frac{1}{i}\frac{dx}{dt}+ix(0)\left[e^{i\xi l}\delta(t-l)-\delta(t)\right],
\end{aligned}
$$
and
$$\aligned
K^{(\xi)}c&=c\cdot \frac{1}{\sqrt 2}[e^{i\xi
l}\delta(t-l)-\delta(t)], \quad (c\in \dC),\\
{K^{(\xi)}}^* x&=\left(x,  \frac{1}{\sqrt 2}[e^{i\xi
l}\delta(t-l)-\delta(t)]\right)=\frac{1}{\sqrt 2}[e^{-i\xi
l}x(l)-x(0)],\\
\endaligned
$$
with $x(t)\in W^1_2$. Here $\bA^{(\xi)}$ is a $(*)$-extension of the operator
$$
T x=\frac{1}{i}\frac{dx}{dt},
$$
 with
 \begin{equation*}
 \dom(T)=\left\{x(t)\,\Big|\,x(t) -\text{ abs. cont.}, x'(t)\in L^2_{[0,l]},\, x(0)=0\right\}.
\end{equation*}
The system of this type was described in details in \cite[Section 8.5]{ABT}. It can also be shown based on this reference that
\begin{equation}\label{e8-62}
   W_{\Theta^{(\xi)}}(z)=1-2i{K^{(\xi)}}^*(\bA^{(\xi)}- z I)^{-1}K^{(\xi)}=e^{i( \xi-z) l}=e^{-i z l}\cdot e^{i\xi l}.
\end{equation}
Set $B^{(\xi)}=e^{i\xi l}$. Then applying \eqref{e6-3-6} we obtain
$$
V_{\Theta^{(\xi)}}(z)=i\frac{W_{\Theta^{(\xi)}}(z)-1}{W_{\Theta^{(\xi)}}(z)+1}=i\frac{B^{(\xi)}e^{-i z l} -1}{B^{(\xi)}e^{-i z l}+1}=i\frac{B^{(\xi)} -e^{i z l}}{B^{(\xi)}+e^{i z l}}.
$$
Note that when $\xi=0$, then $B^{(0)}=1$,  $W_{\Theta^{(0)}}(z)=e^{-i z l}$, and
\begin{equation}\label{e-63-V0}
V_{\Theta^{(0)}}(z)=i\frac{1 -e^{i z l}}{1+e^{i z l}}\quad \textrm{with} \quad V_{\Theta^{(0)}}(i)=i\frac{1 -e^{- l}}{1+e^{- l}}.
\end{equation}

Therefore, $V_{\Theta^{(0)}}(z)\in\sM_{\kappa}$  for $\kappa=e^{-l}$. Comparing \eqref{e8-62} to \eqref{e-35-uni} lets us interpret $B^{(\xi)}=e^{i\xi l}$ as a unimodular transformation of the L-system $\Theta^{(0)}$. In order to find the angle $\alpha$ that corresponds to this unimodular transformation we set $(-e^{2i\alpha})=e^{i\xi l}$ and solve for $\alpha$ to get
\begin{equation}\label{e-63-alpha}
    \alpha=\frac{\xi l-\pi}{2}.
\end{equation}
A controller corresponding to this unimodular transformation is given via \eqref{e-52-Theta0} and is
$$
\Theta_{0,\alpha}=\begin{pmatrix}
          \tan\frac{\xi l-\pi}{2}+i \hspace{2mm} 0    &1  &1    \\
          \dC                   &   &\dC
         \end{pmatrix},
$$
where $\alpha$ is given by \eqref{e-63-alpha} and $\xi l\ne2\pi$. We also have an LF-system
$$
\Theta_{LF}=\Theta^{(0)}\cdot\Theta_{0,\alpha},
$$
that is equivalent to $\Theta^{(\xi)}$ in the sense of Definition \ref{d-14}, that is
$$W_{\Theta_{LF}}(z)=W_{\Theta^{(\xi)}}(z).$$
This LF-system takes form \eqref{e12-29-n} and is explicitly written as
$$
 \Theta_{LF} =\begin{pmatrix}
          \dM \hspace{5.5mm} \dF    &K  &1  \\
          \calH_+ \subset \calH \subset \calH_-                   &   &\dC
         \end{pmatrix},
$$
where
$$
\calH_+ \subset \calH \subset \calH_-=W_1^2\oplus \dC \subset L^2_{[0,l]}\oplus \dC \subset (W^1_2)_-\oplus \dC,
$$
and
$$
\dM=\left(
       \begin{array}{cc}
         \bA^{(0)} & 2i K^{(0)}  \\
         0 & \tan\frac{\xi l-\pi}{2}+i \\
       \end{array}
     \right),
 \quad
 \dF=\left(
     \begin{array}{cc}
                1 & 0 \\
                0 & 0 \\
              \end{array}
            \right),
 \quad
 K=\left(
     \begin{array}{c}
       K^{(0)} \\
       1 \\
     \end{array}
   \right).
$$

\subsection*{Example 2}\label{ex-2}
Now we are going to perturb the function $V_{\Theta^{(0)}}(z)$ in \eqref{e-63-V0} so that it would fall in the class $\sM_\kappa^Q$ for $Q=1$ and $\kappa=e^{-l}$.
We introduce
\begin{equation}\label{e-64-v1}
    V_1(z)=1+i\frac{1 -e^{i z l}}{1+e^{i z l}}.
\end{equation}
Clearly, \eqref{e-63-V0} implies that $V_{1}(z)$ belongs to the class $\sM_\kappa^1$. It can be shown (and checked by direct yet tedious computations) that $V_{1}(z)$ is the impedance function of an L-system of the form
\begin{equation}\label{e6-125}
\Theta_{\rho\mu}=
\begin{pmatrix}
\bA_{\rho\mu}&K &1\\
&&\\
W_2^1\subset L^2_{[0,l]}\subset (W^1_2)_- &{ } &\dC
\end{pmatrix},
\end{equation}
where
\begin{equation}\label{e-107-ex}
\begin{aligned}
\bA_{\rho\mu} x&=i\frac{dx}{dt}+i \frac{1}{\rho+\mu}(\rho x(0)-x(\ell)) \left[\mu\delta(t-\ell)+\delta(t)\right],\\
\bA_{\rho\mu}^* &x=i\frac{dx}{dt}+i \frac{\bar\mu}{\rho+\bar\mu}( x(0)-\rho x(\ell)) \left[\mu\delta(t-\ell)+\delta(t)\right],
\end{aligned}
\end{equation}
$K c=c\cdot \chi$, $(c\in \dC)$, $K^\ast x=(x,\chi)$, $x(t)\in W^1_2$, $\chi=\sqrt{\frac{\rho^2-1}{2|\rho\mu+1|^2}}\,[\mu\delta(t-\ell)-\delta(t)]$.
For the sake of simplicity of further calculations we set $l=\ln 2$. Then the values of parameters $\rho$ and $\mu$ in \eqref{e6-125}-\eqref{e-107-ex} are given by
\begin{equation}\label{e-119-rho-ex2}
\rho=-\frac{343+40\sqrt{13}}{18+45\sqrt{13}},
\end{equation}
and
\begin{equation}\label{e-123-ex2}
    \mu=\frac{1291+25\sqrt{13}+(3087+360\sqrt{13})i}{1291+835\sqrt{13}+(162+405\sqrt{13})i}.
\end{equation}
For the above value of $l=\ln 2$ we have $\kappa=\frac{1}{2}$. Moreover, our function $V_1(z)$ in \eqref{e-64-v1} takes form
\begin{equation}\label{e-68-v1}
    V_1(z)=1+i\frac{1 -2^{i z}}{1+2^{i z}},
\end{equation}
and belongs to the class $\sM_{1/2}^{1}$.
If we want to find a unimodular transformation (and the corresponding controller) that transforms the L-system  $\Theta_{\rho\mu}$ in \eqref{e6-125} into the one whose impedance function  $V_2(z)$ belongs to the class $\sM_{1/2}^{(-1)}$, we apply Theorem \ref{t-11} and formulas \eqref{e-43}-\eqref{e-44}. In our case $Q=1$, and hence $b=a^2$, where
$$
a=\frac{1 -e^{- l}}{1+e^{- l}}=\frac{e^{ l}-1}{e^{ l}+1}=\frac{1}{3},\quad\textrm{ for}\quad l=\ln 2.
$$
Applying \eqref{e-43} gives
$$
\tan\alpha=\frac{b}{2Q}=\frac{a^2}{2}=\frac{1}{18}.
$$
Thus, the value $\alpha=\arctan\frac{1}{18}$ defines the unimodular transformation we seek and provides a controller
$$
\Theta_{0,\alpha}=\begin{pmatrix}
          \frac{1}{18}+i \hspace{2mm} 0    &1  &1    \\
          \dC                   &   &\dC
         \end{pmatrix},
$$
responsible for this transformation in the above sense. Using this value of tangent we obtain
$$
\cos\alpha=\frac{18}{5\sqrt{13}}\quad\textrm{ and }\quad \sin\alpha=\frac{1}{5\sqrt{13}}.
$$
Observe that
$$
V_2(i)=\frac{\cos\alpha+(\sin\alpha)V_1(i)}{\sin\alpha-(\cos\alpha)V_1(i)}=\frac{\frac{18}{5\sqrt{13}}+\frac{1}{5\sqrt{13}}(1+\frac{i}{3})}{\frac{1}{5\sqrt{13}}-\frac{18}{5\sqrt{13}}(1+\frac{i}{3})}=
-\frac{57+i}{51+18i}=-1+\frac{1}{3}i.
$$
This confirms that $V_2(z)\in\sM_{1/2}^{(-1)}$. Finally,
\begin{equation}\label{e-71-v2}
\begin{aligned}
V_2(z)&=\frac{\cos\alpha+(\sin\alpha)V_1(z)}{\sin\alpha-(\cos\alpha)V_1(z)}=\frac{\frac{18}{5\sqrt{13}}+\frac{1}{5\sqrt{13}}(1+i\frac{1 -2^{i z}}{1+2^{i z}})}{\frac{1}{5\sqrt{13}}-\frac{18}{5\sqrt{13}}(1+i\frac{1 -2^{i z}}{1+2^{i z}})}=-\frac{19+i\left(\frac{1 -2^{i z}}{1+2^{i z}}\right)}{17+18i\left(\frac{1 -2^{i z}}{1+2^{i z}}\right)}\\
&=-\frac{19+i+(19-i)2^{iz}}{17+18i+(17-18i)2^{iz}}=-1+\frac{-2+17i-(2+17i)2^{iz}}{17+18i+(17-18i)2^{iz}}.
\end{aligned}
\end{equation}

We have shown that applying a unimodular transformation with $\tan\alpha=1/18$ maps function $V_1(z)\in\sM_{1/2}^{1}$ of the form \eqref{e-68-v1} into a function $V_2(z)\in\sM_{1/2}^{(-1)}$ of the form \eqref{e-71-v2}.

\appendix
\section{Differential Equations and L- and F-systems}\label{A1}

Let $T\in\Lambda$, 
$K$ be a bounded 
linear operator from a finite-dimensional Hilbert space $E$ into
$\calH_-$, $K^*\in [\calH_+,E]$, and $J=J^\ast =J^{-1}\in [E,E]$.
Consider the  following singular system of equations
\begin{equation}\label{e6-37-new-T}
     \left\{%
\begin{array}{ll}
    {i}\frac{d\chi}{dt}+T\chi(t)=KJ\psi_-(t), & \hbox{} \\
    \chi(0)=x\in\dom(T),&\\
    \psi_+=\psi_- -2iK^* \chi(t). & \hbox{} \\
\end{array}%
\right.
\end{equation}
Given an input vector $\psi_-=\varphi_- e^{i{z} t}\in E$, we seek
solutions to the system \eqref{e6-37-new-T} as an output vector
$\psi_+=\varphi_+ e^{i{z} t}\in E$, and a state-space vector
$\chi(t)=xe^{i{z} t}\in \dom(T)$. Substituting the expressions for
$\psi_\pm(t)$ and $\chi(t)$
 allows us to cancel exponential terms and convert
the system \eqref{e6-37-new-T} to the form
\begin{equation}\label{e6-3-1}
\begin{cases}
\begin{aligned} (T&-zI)x=KJ\varphi_-,\\
\varphi_+&=\varphi_--2iK^\ast x,
\end{aligned}
\end{cases}\quad z\in\rho(T).
\end{equation}
The choice of the operator $K$ in the above system is such that $KJ\varphi_-\in\calH_-$. Therefore the first equation of \eqref{e6-3-1} does not, in general, have a regular solution $x\in\dom(T)$. It has, however, a generalized solution $x\in\calH_+$ that can be obtained in the following way. If $z\in\rho(T)$, then we can use the density of $\calH$ in $\calH_-$ and therefore there is a sequence of vectors $\{\alpha_n\}\in\calH$ that approximates $KJ\varphi_-$ in $(-)$-metric. In this case the state space vector $x=\hat R_z(T)KJ\varphi_-\in\calH$ is understood as $\lim_{n\to\infty}(T-zI)^{-1}\alpha_n$, where $\hat R_z(T)$ is the extended to $\calH_-$ by $(-,\cdot)$-continuity resolvent $(T-zI)^{-1}$. But then we can apply \cite[Theorem 4.5.9]{ABT} to conclude that $x\in\calH_+$. This explains the expression $K^*x$ in the second line of \eqref{e6-3-1}. In order to satisfy the condition $\IM T=KJK^*$ we perform the \textit{regularization} of system \eqref{e6-3-1} and use $\bA\in[\calH_+,\calH_-]$,  a $(*)$-extension of $T$ such that $\IM\bA=KJK^*$. This leads to the system
\begin{equation}\label{e6-3-1-T}
\begin{cases}
\begin{aligned} (\bA&-zI)x=KJ\varphi_-,\\
\varphi_+&=\varphi_--2iK^\ast x,
\end{aligned}
\end{cases}\quad z\in\rho(T),
\end{equation}
where $\varphi_- $ is an input vector, $\varphi_+$ is an output
vector, and $x$ is a state space vector of the system. System
\eqref{e6-3-1-T} is the stationary version of the system
\begin{equation}\label{e6-37-new-}
     \left\{%
\begin{array}{ll}
    {i}\frac{d\chi}{dt}+\bA\chi(t)=KJ\psi_-(t), & \hbox{} \\
    \chi(0)=x\in\calH_+,&\\
    \psi_+=\psi_- -2iK^* \chi(t). & \hbox{} \\
\end{array}%
\right.
\end{equation}
Both differential equation systems \eqref{e6-3-1-T} and \eqref{e6-37-new-} are associated with the corresponding L-system $\Theta$ of the form \eqref{e6-3-2}.

Similar connections can be built for F-systems. Let $M$ be a bounded linear operator in $\calH$  and let $F$ be an orthogonal projection in $\calH$, $K\in[E,\calH]$, and $J$ be a bounded, self-adjoint, and unitary operator in $E$. Let also $\IM M=KJK^*$ and $L^2_{[0,\tau_0]}(E)$ be the Hilbert space of $E$-valued functions equipped with an inner product
\begin{equation*}\label{e12-12}
(\varphi,\psi)_{L^2_{[0,\tau_0]}(E)}=
\int_0^{\tau_0}(\varphi,\psi)_E\,dt, \quad
\left(\varphi(t),\,\psi(t)\in L^2_{[0,\tau_0]}(E)\right).
\end{equation*}
 Consider the following system of equations
\begin{equation}\label{e12-16}
     \left\{%
\begin{array}{ll}
    iF\frac{d\chi}{dt}+M\chi(t)=KJ\psi_-(t), & \hbox{} \\
   \chi(0)=x\in\calH, &\\
    \psi_+=\psi_- -2iK^* \chi(t). & \hbox{} \\
\end{array}%
\right.
\end{equation}
Given an input vector $\psi_-=\varphi_- e^{i{z} t}\in E$, we seek solutions to the system \eqref{e12-16} as an output vector $\psi_+=\varphi_+ e^{i{z} t}\in E$ and a state-space vector $\chi(t)=xe^{i{z} t}\in\calH$. Substituting the expressions for $\psi_\pm(t)$ and $\chi(t)$  allows us to cancel exponential terms and convert the system \eqref{e12-16} to the stationary form
\begin{equation}\label{e12-18'}
    \left\{%
\begin{array}{ll}
    (M-zF)x=KJ\varphi_-, & \hbox{} \\
    \varphi_+=\varphi_- -2iK^* x, & \hbox{} \\
\end{array}%
\right.\quad z\in\rho(M,F).
\end{equation}
Both differential equation systems \eqref{e12-16} and \eqref{e12-18'} are associated with the corresponding F-system $\Theta_F$ of the form \eqref{system1}.

It can be shown in \cite{ABT} that L-systems written in the form \eqref{e6-3-1-T} (or \eqref{e6-37-new-}) and F-systems written in the form \eqref{e12-16} (or \eqref{e12-18'}) obey appropriate conservation laws. For details the reader is referred to Sections 6.3 and 12.1 of \cite{ABT}.


\end{document}